\documentclass[letterpaper]{article}
\usepackage[utf8]{inputenc}

\title{On the scaling of  random Tamari intervals and Schnyder woods of random triangulations (with an asymptotic D-finite trick)}

\author{Guillaume Chapuy%
\thanks{Université Paris Cité, CNRS, IRIF, F-75013, Paris, France.
Email:~{\tt guillaume.chapuy@irif.fr}. 
}
}

\date{\today}

\usepackage{graphicx, enumitem}
\usepackage{amsmath,amsfonts,amssymb}
\usepackage{amsthm, xcolor}
\usepackage{comment, fullpage, placeins,hyperref}

\usepackage[margin=7mm]{caption}

\theoremstyle{theorem}
\newtheorem{theorem}{Theorem}
\newtheorem{proposition}[theorem]{Proposition}
\newtheorem{corollary}[theorem]{Corollary}
\newtheorem{lemma}[theorem]{Lemma}

\numberwithin{theorem}{section}

\theoremstyle{definition}

\newtheorem{remark}[theorem]{Remark}

\newtheorem{method}[theorem]{Method}
\newtheorem{example}[theorem]{Example}

\newcommand{\singeq}{\hat{\sim}}

\begin{document}

\maketitle

\begin{abstract} 
	We consider a Tamari interval of size $n$ (i.e., a pair of Dyck paths which are comparable for the Tamari relation) chosen uniformly at random. We show that the height of a uniformly chosen vertex on the upper or lower path scales as $n^{3/4}$, and has an explicit limit law.
By the Bernardi-Bonichon bijection, this result also describes the height of points in the canonical Schnyder trees of a uniform random plane triangulation of size $n$.

	The exact solution of the model is based on polynomial equations with one and two catalytic variables. 
	To prove the convergence  from the exact solution, we use a version of moment pumping based on D-finiteness, which is essentially automatic and should apply to many other models. We are not sure to have seen this simple trick used before. 

	It would be interesting to study the universality of this convergence  for decomposition trees associated to positive Bousquet-Mélou--Jehanne equations.

\end{abstract}

\section{Introduction and main results}

\begin{figure}[h]
	\centering
	\includegraphics[width=0.7\linewidth]{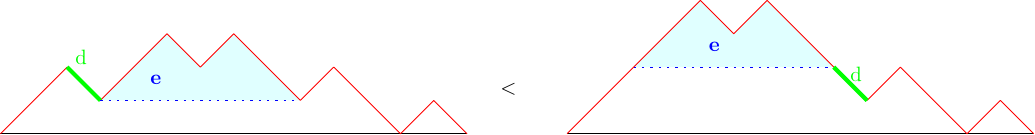}
	\caption{A Dyck path of size $7$, with a marked down step followed by an up-step. The path obtained by flipping this down-step with the shortest excursion following it is declared to be \emph{larger}. The Tamari lattice is the partial order generated by all such relations.}\label{fig:defTamari}
\end{figure}

For $n\geq 0$, a \emph{Dyck path of size $n$} is a lattice path made of $n$ up-steps and $n$-down steps, starting (and ending) at height $0$, and whose height stays always nonnegative. See Figure~\ref{fig:defTamari}.
Dyck paths of size $n$ are in immediate bijection with well-formed parenthesis words of length $2n$, and they are counted by the Catalan number
$$
Cat(n):=\frac{1}{n+1} {2n \choose n}.
$$
The set $\mathcal{D}_n$ of Dyck paths of size $n$ can be endowed with several interesting (partial) orders. Maybe the most interesting and natural one is the \emph{Tamari partial order}, whose covering relation is nothing but the edge-flip once interpreted through the classical bijection between Dyck paths and triangulations of a polygon. The Tamari partial order is a lattice, whose covering relation can be described as follows: Let $P$ be a Dyck path, and let $d$ be a down-step in $P$, followed by an up-step. Let $\mathfrak{e}$ be the shortest excursion following $d$ in $P$ (an excursion is a path staying higher than its starting point except for its last point). Then the Dyck path $Q$ obtained from $P$ by exchanging $d$ and $\mathfrak{e}$ is declared \emph{larger} than $P$ for the Tamari order. The reflexive transitive closure of this relation defines the Tamari lattice. See Figure~\ref{fig:defTamari}.

The Tamari lattice plays an important role in many facets of algebraic combinatorics and discrete mathematics. To name a few: its Hasse-diagram is the 1-skeleton of a famous polytope called the associahedron, see e.g.~\cite{TamariFestschrift,PilaudSantosZiegler}; determining its diameter has been a longstanding problem solved only recently~\cite{diameterPournin}, despite mysterious connections to hyperbolic geometry~\cite{diameterTarjan}; and, quite excitingly, determining  the mixing time on this graph is still an open problem despite recent partial results~\cite{McShineTetali,Eppstein2022}.

In this paper we are interested in yet another connection, related to \emph{Tamari intervals}. Following the standard combinatorics terminology, we define  an \emph{interval} of size $n$ in the Tamari lattice as a pair $(P,Q) \in (\mathcal{D}_n)^2$ such that $P\leq Q$ (for the Tamari partial order).
We let $\mathcal{I}_n$ be the set of Tamari intervals of size $n$.
In a famous paper, Chapoton~\cite{Chapoton} proved that the number of elements of $\mathcal{I}_n$ was given by 
\begin{align}\label{eq:Chapoton}
|\mathcal{I}_n| = \frac{2}{n(n+1)}{4n+1 \choose n-1},
\end{align}
which is also the number of rooted planar triangulations\footnote{defined as embedded planar graphs with only faces of degree $3$. These objects are much more complex than the triangulations~\emph{of a single polygon} Dyck paths are in bijection with.} of size $n$~\cite{Tutte:censusTriangulations}. An elegant, and deep, direct bijective proof has been given by Bernardi and Bonichon in~\cite{BernardiBonichon}. Since Chapoton's discovery, the analogy between Tamari intervals and planar maps has been much developped.
One the one hand, the enumeration of Tamari intervals can be refined (or weighted), giving rise to product formulas~\cite{BMCPR} surprisingly similar to the ones counting planar constellations~\cite{BMS}, suggesting the existence of an underlying object maybe as rich as the GUE 1-matrix model, with a potentially integrable structure, which is yet to be discovered.
On the other hand, many different variants of Tamari intervals have been introduced and enumerated. Their enumeration exhibits a universality phenomenon very similar to the one existing in the world of planar maps, and in fact many of these variants are equi-enumerated with some well-known family of planar maps. The understanding of this phenomenon is still very partial, see~\cite{Fang} and references therein.
More recently, non-lattice variants of the Tamari partial order have been introduced for which the similarity to maps is even more stunning~\cite{BMChapoton}.

\subsection{Main results}

Although large random planar maps have been intensely studied in the last decades (see for example, among hundreds of references,~\cite{ChassaingSchaeffer, BDFG:distances, LeGallsurvey}, it seems that the behaviour of random Tamari intervals has not been studied (and we are not sure to know why). It is however natural to ask this question:
\begin{quote}
{\it What does a large uniformly random Tamari interval look like?} 
\end{quote}
\noindent We are not aware of any results in this direction.  In this paper we give a first answer to this question. 
If $P\in \mathcal{D}_n$ and $i\in [0..2n]$, we write $P(i)$ for the height of the point of $P$ lying at abscissa $i$.
We show:
\begin{theorem}\label{thm:tamari}
	Let $(P,Q)$ be a Tamari interval of size $n$, chosen uniformly at random in $\mathcal{I}_n$.
	Let $I\in[0..2n]$ be an integer chosen uniformly at random, and let $Q_n(I)$ be the height of the point of the upper path $Q$ lying at abscissa $I$. Then we have the convergence in law
	\begin{align}\label{eq:limitlaw}
	\frac{Q_n(I)}{n^{3/4}} \longrightarrow Z
	=(XY)^{1/4}
	\end{align}
	when $n$ goes to infinity, where $X\sim \beta(\frac{1}{3},\frac{1}{6})$ and $Y\sim \Gamma(\frac{2}{3},\frac{1}{2})$ are independent random variables. In fact, we have the convergence of all moments: for integer $k\geq 0$,
	\begin{align}\label{eq:mainMoment}
		\mathbf{E}\left[ \left(\frac{Q_n(I)}{n^{3/4}}\right)^k \right]\longrightarrow \frac{\sqrt{3} \cdot2^{-\frac{k}{4}-1}}{\sqrt{\pi}} \frac{\Gamma(\frac{1}{4}k+\frac{1}{3})\Gamma(\frac{1}{4}k+\frac{2}{3})}{\Gamma(\frac{1}{4}k+\frac{1}{2})}.
	\end{align}
\end{theorem}
We recall that the random variables $\beta(a,b)$ and $\Gamma(\ell, \theta)$ have respective densities
$\frac{\Gamma(a+b)}{\Gamma(a)\Gamma(b)}x^{a-1}(1-x)^{b-1}$ on $(0,1)$
and 
$
\frac{x^{k-1}e^{-x/\theta}}{\Gamma(k)\theta^k}	
$ on $\mathbb{R}_+$. Their respective $k$-th moments are
$\frac{\Gamma(a+b)\Gamma(a+k)}{\Gamma(a)\Gamma(a+b+k)}$, and  $\frac{\theta^k\Gamma(l+k)}{\Gamma(\ell)}$, so it is direct to check that~\eqref{eq:mainMoment} with $k$ substituted by $4k$ is indeed equal to the $k$-th moment of $XY$. 

Interestingly, the random variable $Z^4$ already appears as a limit law in a (seemingly unrelated) physics context~\cite{physics}. This reference also gives a closed form for the density (but does not identify it as a product of a Gamma and a Beta variable)\footnote{I thank Philippe Biane for help identifying the variable from its moments, and Thomas Budzinski for suggesting to look for them in the OEIS, which is how I found the reference~\cite{physics}, see the sequence OEIS:A064352. Note that using the duplication formula the R.H.S. of~\eqref{eq:mainMoment} can also be written $\frac{\Gamma(3k/4)}{\Gamma(k/2)}2^{-1+k/4}3^{1-3k/4}$, see~\cite{maple}, and see also~\eqref{eq:momentstwoexpr} for yet another expression.}.
\begin{figure}
	\hspace{1cm}
	\includegraphics[width=0.4\linewidth]{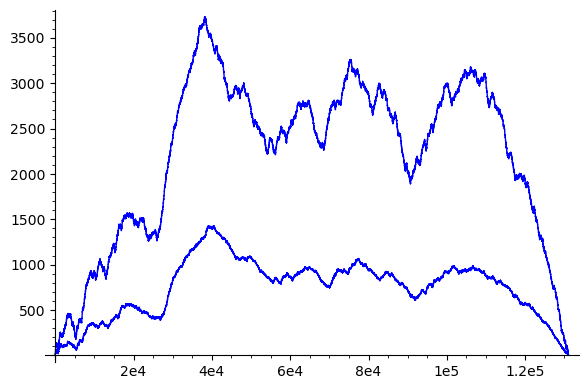}
	\hfill
	\includegraphics[width=0.4\linewidth]{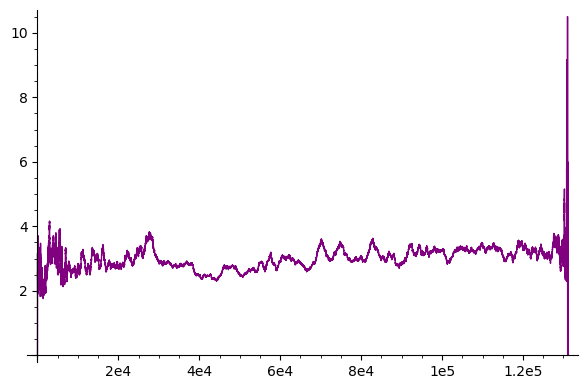}
	\hspace{1cm}
	\caption{Left: A  uniform random Tamari interval $(P_n,Q_n)$ of size $n=65536$ generated with a python code generously  provided by Wenjie Fang. Right: plot of $Q_n(i)/P_n(i)$.}\label{fig:randomPicture}
\end{figure}

\medskip

One can also address the case of the lower path, at the price of more technical, and maybe less standard, methods. We obtain the following theorem.
\begin{theorem}\label{thm:tamariLow}
	Let $(P,Q)$ be a Tamari interval of size $n$, chosen uniformly at random in $\mathcal{I}_n$. 	Let $I\in[0..2n]$ be an integer chosen uniformly at random, and let $P_n(I)$ 
	be the height of the point of the lower path $P$
	lying at abscissa $I$. Then we have the convergence in law
	\begin{align}\label{eq:limitlawLow}
		\frac{P_n(I)}{n^{3/4}} \longrightarrow \frac{1}{3}Z.
	\end{align}
	when $n$ goes to infinity, where $Z$ is an in Theorem~\ref{thm:tamari}. Again, we have the convergence of all moments.
\end{theorem}
%
In view of Theorems~\ref{thm:tamari} and~\ref{thm:tamariLow}, it is natural to suspect that
\begin{align}\label{eq:conjThird}
P_n(I) = \frac{1}{3} Q_n(I) + o(Q_n(I)),
\end{align}
with a little-o in probability, i.e. that, roughly speaking, the lower path is almost a third of the lower path, asymptotically almost surely and almost everywhere. See Figure~\ref{fig:randomPicture} which shows a simulation for $n=10^5$, which supports this conjecture.

Unfortunately our methods based on functional equations make it hard to track the joint law of  $(P_n(I),Q_n(I))$ so we cannot prove~\eqref{eq:conjThird}. However, we can prove a result of similar strength for a slight variant of the problem. For a Dyck path $P$ of size $n$ and an integer $j\in [n]$, we let $\tilde P(j)$ be the initial height of the $j$-th up-step of $P$.
It turns out that we can write a functional equation for the generating function corresponding to the joint law of $(\tilde{P}_n(J),\tilde{Q}_n(J))$ where $J$ is uniform on $[1,n]$. We can prove:
\begin{theorem}\label{thm:mixed}
	Let $J$ be uniform on $[n]$, and $(P,Q)$ be uniform on $\mathcal{I}_n$. Then one has, when $n$ goes to infinity,  
	\begin{align}\label{eq:mixedSecondMoment}
		\mathbf{E}[(\tilde{Q}_n(J)-3\tilde{P}_n(J))^2]=O(n).
	\end{align}
In particular, we have, in probability, 
\begin{align}\label{eq:conjThirdTilde}
\tilde{P}_n(J) = \frac{1}{3} \tilde{Q}_n(J) + o(\tilde{Q}_n(J)).
\end{align}
\end{theorem}

Note that if one is interested only in the individual convergence, studying $P_n(I)$ and $\tilde{P}_n(J)$ does not make much of a difference, thanks the following well-known lemma:
\begin{lemma}\label{lemma:coupling}
	Let $P_n$ be a fixed Dyck path of size $n$. Then it is possible to couple a uniform  variable $J\in [n]$ and a uniform variable $I\in[0,2n-1]$ such that 
	$$
|P_n(I)-\tilde{P}_n(J) |\leq 1.
	$$
\end{lemma}
\begin{proof}
First sample $J\in [n]$ uniformly. Let $I_1$ be the initial abscissa of the $J$-th up step of $P$, and let $I_2$ be the initial abscissa of the down-step that matches this step, viewing $P$ as a parenthesis word. Then let $I$ be equal to $I_1$ or $I_2$ with probability $\frac{1}{2}$. It is clear that $I$ is uniform on $[0,2n-1]$ -- the converse operation consists in considering the step starting at $I$, which is part of a pair of up and down steps matched together in the underlying parenthesis word, and letting $J$ be the index of the up-step.
\end{proof}
 From the lemma and the fact that the case $I=2n$ has  a negligible contribution to moments, the individual convergences (in distribution and of all moments) 
$$\frac{\tilde{P}_n(J)}{n^{3/4}} \rightarrow \frac{Z}{3}
\ \ , \ \ 
\frac{\tilde{Q}_n(J)}{n^{3/4}} \rightarrow Z$$ 
are equivalent respectively to the one of $P_n(I)$ and $Q_n(I)$. Note however that this coupling argument does not apply to their joints laws, which is why 
\eqref{eq:conjThird} and~\eqref{eq:conjThirdTilde} are a priori incomparable statements. However we could deduce one from the other if we had some control on the regularity of the paths in their natural scaling limit.

As we said, Bernardi and Bonichon~\cite{BernardiBonichon} provided an explicit bijection between intervals in $\mathcal{I}_n$ and rooted plane triangulations of size $n$. Such a triangulation can always be equipped in a canonical way with a \emph{realizer} or \emph{Schnyder wood}, which is a partition of its internal edges into three subsets (say red, blue, green), such that the edges of each subset form a tree, with certain conditions. See Figure~\ref{fig:Schnyder}. In the Bernardi-Bonichon bijection, if the interval $(P,Q)$ is in correspondence with the triangulation $T$, then $P$ encodes the blue tree in the canonical Schnyder wood of $T$. Therefore we have:
\begin{corollary}\label{cor:schnyder}
	Let $T_n$ be a rooted plane triangulation of size $n$, chosen uniformly at random, and let $(T^{(1)}_n, T^{(2)}_n, T^{(3)}_n)$ be its canonical Schnyder wood, that is to say, the one associated to its minimal orientation in the sense of~\cite{BernardiBonichon}. Let $V$ be a uniform random internal vertex of $T_n$ and let $H_n^{(i)}$ be the height of the vertex $V$ in the tree $T^{(i)}_n$. Then, for any $i\in [3]$ we have
	$$
	\frac{H_n^{(i)}}{n^{3/4}} \longrightarrow \frac{1}{3} Z,
	$$
	where $Z$ is in as in Theorem~\eqref{thm:tamari}.
\end{corollary}
\begin{proof}
%
%
%
	The Bernardi-Bonichon bijection implies that $H_n^{(1)}$ has the same law as $1+\tilde{P}_n(J)$ in previous notation. This implies the result for $i=1$, and by symmetry for $i\in [3]$. 
\end{proof}

\begin{figure}
\centering	\includegraphics[width=0.6\linewidth]{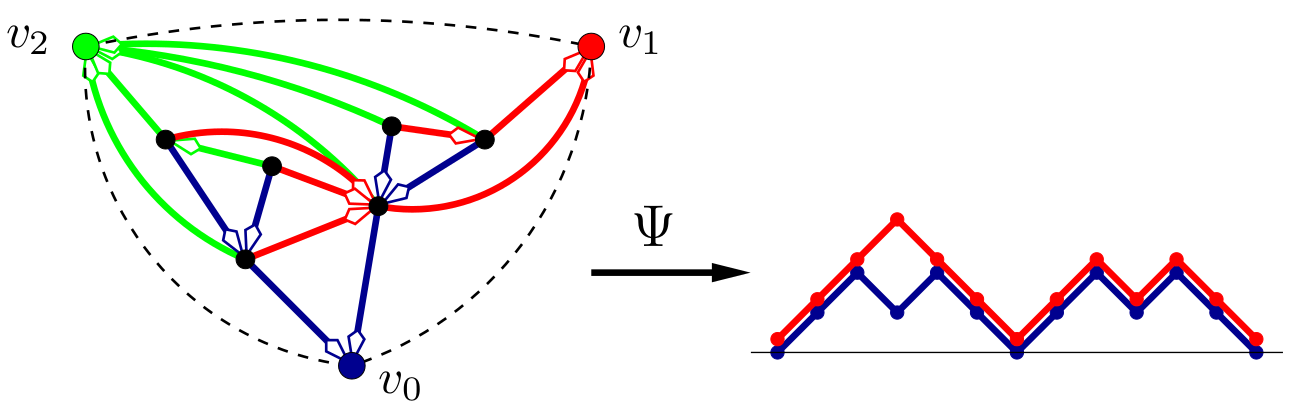}
	\caption{A figure taken from~\cite{BernardiBonichon} (thanks to the authors). A rooted planar triangulation equipped with its minimal Schnyder-wood, and its image (a Tamari interval) by the Bernardi-Bonichon bijection. The  lower path is nothing but the contour function of the blue tree.
}\label{fig:Schnyder}
\end{figure}

\medskip

The proofs Theorems~\ref{thm:tamari},~\ref{thm:tamariLow}, and~\ref{thm:mixed} each have two parts: the first one consists in solving "exactly" the model, by obtaining an explicit algebraic equation for the generating function $f(t,s)$ of Tamari intervals having a marked abscissa, with control on the size $n$, and the random variable we want to control. 
This requires to solve an equation with catalytic variables, which is specific to each case. The second part is to deduce the asymptotic of moments from there, which is a problem of analytic combinatorics in 2 variables, for which we need to find good tools. We describe below a simple method that will do the trick, and that we believe is of independent interest.

\subsection{A method to prove the convergence of moments from an algebraic equation}
\label{sec:trick}

Assume that we have a combinatorial class $\mathcal{A}$ equipped with a size function $|\cdot|$ and an integer statistic $\mu$, and consider the random variable $X_n=\mu(A_n)$ where $A_n$ is an object of size $n$ in $\mathcal{A}$ chosen uniformly at random. Consider the generating function
\begin{align}\label{eq:deffts}
f(t,s) = \sum_{n\geq 0} \sum_{p \in \mathbb{Z}}a_{n,p} t^n s^p
= \sum_{n\geq 0} \sum_{p \in \mathbb{Z}} a_n \mathbf{P}[X_n=p] t^n s^p,
\end{align}
with $a_{n,p} = |\{\alpha \in \mathcal{A}, |\alpha|=n, \mu(\alpha)=p\}|$, and  $a_{n} = |\{\alpha \in \mathcal{A}, |\alpha|=n\}|$.

Let us assume that the generating function $f(t,s)$ 
is algebraic, i.e. there exists a nontrivial polynomial $A$ such that 
$$
A(f(t,s),t,s)=0.
$$
%
For $k\geq 0$ we consider the generating function of factorial moments
$$f^{(k)}\equiv f^{(k)}(t) :=\left.\left(\left(\frac{d}{ds}\right)^k f(t,s)\right)\right|_{s=1}
%
%
=
\sum_{n\geq 0} a_n \mathbf{E}[ (X_n)_{(k)} ]t^n,
$$
with $(X_n)_{(k)}:= X_n (X_n-1)\dots (X_n-k+1)$.
To study the convergence of the random variable $X_n$, a standard way called the method of moments consists in studying the asymptotics for fixed $k$ of its $k$-th moment $\mathbf{E}[ (X_n)^k ]$ (or factorial moment $\mathbf{E}[ (X_n)_{(k)} ]$). In our setting, this is equivalent to studying the asymptotics of the coefficient $[t^n]f^{(k)}(t)$, for fixed $k$. 
Now, since all the functions $f^{(k)}(t)$ are algebraic, they are amenable to singularity analysis in the sense of~\cite{FO, FS}, as we will recall now.

In what follows, for a function $g$ and $\alpha, c \in \mathbb{R}$, $\rho>0$ we write
$$
g(t)\singeq c (1-t/\rho)^\alpha
	\ \ , \ \ t\rightarrow \rho,
$$
if $g(t)$, as an analytic function, has no singularity on the closed circle of radius $\rho$ except maybe at $t=\rho$, and if $g(t)$ has a Puiseux expansion of the form
$$
g(t) = P(t) + c (1-t/\rho)^\alpha + o((1-t/\rho)^\alpha)
$$
near $t=\rho$, where $P$ is a polynomial.
Let us now recall the classical transfer theorem, slightly reformulated\footnote{We will consider only the case of a unique singularity of minimal modulus. See Remark~\ref{rem:general}}. Note that we allow $c=0$ in the statement, and also $\alpha\in \mathbb{N}$ provided $c=0$. 
\begin{proposition}[Transfer theorem for algebraic functions~\cite{FO,FS}]\label{prop:transfer}
	Let $g(t)=\sum_{n\geq 0} g_n t^n$ be an algebraic function, with radius of convergence $\rho>0$, and assume that $t=\rho$ is the unique singularity of $g$ of minimal modulus. Assume that $g(t)$ has a singular expansion of the form
\begin{align}
	g(t) \singeq c (1-t/\rho)^\alpha
	\ \ , \ \ t\rightarrow \rho,
	\label{eq:singg}
\end{align}
	with $c\in \mathbb{R}$ and $\alpha \in \mathbb{R}$, and $c=0$ if $\alpha \in \mathbb{N}$.
	Then the coefficients $g_n=[t^n]g(t)$ satisfy the asymptotics, when $n\rightarrow \infty$,
	$$
	\frac{g_n}{n^{-1-\alpha}\rho^{-n}}\rightarrow \frac{c}{\Gamma(-\alpha)}  .
	$$
\end{proposition}
Of course, when $\alpha\in \mathbb{N}$ and $c=0$ we interpret the right-hand side as zero.

Thanks to the transfer theorem, to estimate moments we only need to determine the expansion~\eqref{eq:singg} for each of the functions $g=f^{(k)}$.
Here comes the main trick: since the function $f(t,s)$ is algebraic, it is also D-finite, i.e. the coefficients of its expansion in any variable satisfy a linear equation with polynomial coefficients (see e.g.~\cite{Comtet} or~\cite{Stanley}). We will apply this to the coefficients of the expansion\footnote{The function $f(t,1+r)$ is algebraic, therefore it is $D$-finite in the variable $r$. No notion of convergence is required to say this. Of course, one has to be careful about which branch of this function one considers when performing that actual calculations.} in the variable $r$ such that $s=r+1$. These coefficients are, up to a factorial factor, the functions $f^{(k)}$, i.e.
$
[r^k] f(t,1+r) = \frac{1}{k!} f^{(k)}(t).
$
%
Therefore, we can write a recurrence formula of the form:
\begin{align}\label{eq:Dfinite}
	\sum_{d=0}^L   P_{d}(t,k) \frac{f^{(k-d)}(t)}{(k-d)!}=0 
\ \ , \ \  
k \geq L,
\end{align}
with $L\geq 1$ and $P_0, \dots, P_L$ being bivariate polynomials over $\mathbb{Q}$.
In practice, it will be useful to allow ourselves to work under some change of variable (of the variable $t$), and for this reason we will more generally assume that our algebraic function $f$ satisfies~\eqref{eq:Dfinite} with $P_0, \dots, P_L$ being polynomials in $k$ whose coefficients are algebraic functions 
of $t$ (this will be the case in our main applications of the method).

To sum up, one can compute the $f^{(k)}$ by induction, using
\begin{align}\label{eq:Dfinite2}
	f^{(k)}(t)= \sum_{d=1}^{L} h_d(t,k)  f^{(k-d)}(t) 
\ \ , \ \  
k \geq L,
\end{align}
with $h_d(t,k)=-\frac{k!P_{d}(t,k)}{(k-d)!P_0(t,k)}$, which is a rational function in $k$ whose coefficients are algebraic functions of~$t$. This leads us to:
\begin{method}[D-finite trick for moment pumping]
	Given a bivariate  algebraic function $f$, obtain a linear equation for its derivatives $f^{(k)}$ of the form~\eqref{eq:Dfinite2} using standard computer algebra tools. Then,  use it to determine the asymptotic of $f^{(k)}$ near $t=\rho$ by induction on $k$. Deduce the asymptotics of $[t^n]f^{(k)}$ using the transfer theorem (Proposition~\ref{prop:transfer}).
\end{method}

The idea of this method is quite general, but of course one has to be careful to carry the analytical details in the induction (check that no unwanted singularity appears, for example; or be careful about multiple dominant singularities). Let us give a simple framework of application in the case of a unique dominant singularity in the next theorem,  whose proof is essentially immediate.


\begin{theorem}[D-finite trick for moment pumping, an instance]\label{thm:trick}
	Let $f(t,s)$ be a generating function of the form~\eqref{eq:deffts}, and assume that $f(t,s)$ is an algebraic function. Then $f$ is D-finite in the variable $s-1$, therefore it satisfies an equation of the form~\eqref{eq:Dfinite2}, with $L\geq 1$ and where for $d\in [L]$, $h_d(t,k)$ is a rational function of $k$ whose coefficients are algebraic functions of $t$. Assume that:

	(i)  There is $\beta>0$ such that for each $d\in [L]$ ,
	$$h_d(t,k)(1-t/\rho)^{\beta d} \rightarrow a_d(k)  \ \ , \ \ t\rightarrow \rho,$$ with possibly $a_d(k)=0$. Moreover, for $k\geq 1$, $h_d(t,k)$ has no other singularity than $\rho$ on the closed disk of radius $\rho$.

	(ii) "Initial conditions": there is $\alpha \in \mathbb{R}\setminus\mathbb{N}$ and numbers $c_\ell$, with $c_0>0$, such that 
	$$f^{(\ell)}\singeq c_\ell (1-t/\rho)^{\alpha-\beta\ell}$$
	for all $0\leq \ell\leq \ell_0$, with $\ell_0:=L+ \max(\lfloor\alpha/\beta\rfloor,-1)$.
	Moreover, $c_{\ell}=0$ if $\alpha-\beta\ell \in \mathbb{N}$.

	Then one has 
	$f^{(k)}\singeq c_k (1-t/\rho)^{\alpha-\beta k}$
	for any $k \geq 0 $, where $c_k$ is given by the recurrence formula
	\begin{align}\label{eq:ck}
	c_{k}= \sum_{d=1}^{L} a_d(k)  c_{k-d} \ \  ,\ \  k> \ell_0,
\end{align}
	and the values of $c_k$ for $k\leq \ell_0$ are given by the initial conditions.
Moreover, for each $k\geq 0$,
\begin{align}\label{eq:trickmoment}
	\frac{\mathbf{E}[(X_n)^k]}{n^{\beta k}} \rightarrow \frac{c_k\Gamma(-\alpha)}{c_0\Gamma(\beta k-\alpha)}.
\end{align}
\end{theorem}
\begin{proof}
	We will show by induction that $f^{(\ell)}(t)$ has a unique singularity of minimal modulus at $t=\rho$, with $f^{(\ell)}=c_\ell (1-t/\rho)^{\alpha-\beta\ell} + o((1-t/\rho)^{\alpha-\beta\ell})$, for every $\ell \geq 0$ with $\ell>\frac{\alpha}{\beta}$.

	For $\frac{\alpha}{\beta}<\ell \leq \ell_0$, this is true by the initial conditions (note that $\alpha-\beta\ell<0$ for these values of $\ell$, so the polynomial part in the definition of the symbol $\singeq$ has a subdominating contribution).
	
	Now let $k \geq \ell_0+1$ and assume the induction hypothesis for $\frac{\alpha}{\beta}<\ell<k$. Note that for all $d\in [L]$, we have 
$$k-d \geq \ell_0+1-L \geq \lfloor\frac{\alpha}{\beta}\rfloor +1> \frac{\alpha}{\beta}
$$ 
we can thus use the induction hypothesis for all the functions $f^{(k-d)}$ appearing in the right-hand side of~\eqref{eq:Dfinite2}. Using~\eqref{eq:Dfinite2}, the hypothesis on the singularity of $h_d$, and induction, it then immediately follows that
	$$
	f^{(k)}=c_k (1-t/\rho)^{\alpha-\beta k} + o((1-t/\rho)^{\alpha-\beta k})
	$$
	with $c_k$ given by~\eqref{eq:ck}, and moreover that $f^{(k)}$ has a unique dominant singularity at $\rho$.
This concludes the induction, and shows that we have $f^{(\ell)}\singeq c_\ell (1-t/\rho)^{\alpha-\beta\ell}$ for all $\ell\geq0$.

	By the transfer theorem, this implies that $[t^n]f^{(\ell)} \sim \frac{c_\ell}{\Gamma(\beta\ell-\alpha)}n^{\beta\ell-\alpha-1}\rho^{-n}$.
	Since $\mathbf{E}[(X_n)_{(k)}]=\frac{[t^n]f^{(k)}}{[t^n]f^{(0)}}$,
we obtain from the transfer theorem that
$$
	\frac{\mathbf{E}[(X_n)_{(k)}]}{n^{\beta k}} \rightarrow \frac{c_k\Gamma(-\alpha)}{c_0\Gamma(\beta k-\alpha)}.$$
	Since $\beta>0$, the asymptotics~\eqref{eq:trickmoment} follows from the triangular relation between moments and factorial moments.
\end{proof}

\begin{remark}\label{rem:general}
	We do not try to provide minimal hypotheses in this theorem: our goal is to emphasize the method, that seems applicable to various situations, including some D-finite non algebraic functions. The only requirement for the method to work is that dominant singularities of $f^{(k)}$ are "not too hard" to track by induction from the linear recurrence formula~\eqref{eq:Dfinite2}.
	Note also that there is  no direct objection to trying to apply the method for more than bi-variate examples, i.e. for the asymptotic of joint moments of variables of the form $\mathbf{E}X_{n,1}^{k_1}\dots X_{n,j}^{k_j}$ for fixed $k_1,\dots,k_j$ when $n\rightarrow \infty$, provided the corresponding $(j+1)$-variate generating function is algebraic.
\end{remark}

\begin{remark}
	We insist that Theorem~\ref{thm:trick} if applicable, is essentially automatic. Indeed, computer algebra softwares (e.g. the Maple package gfun~\cite{gfun}) are able to provide a recurrence of the form~\eqref{eq:Dfinite2} automatically from an algebraic equation for $f$. Also, since all the $f^{(k)}$ are algebraic with a computable algebraic equation, it is in principle automatic to check the initial conditions required in the theorem for $0\leq \ell \leq \ell_0$. All this will be illustrated by the proofs of Theorems~\ref{thm:tamari}-\ref{thm:tamariLow}.
\end{remark}

\begin{example}\label{eq:Dyck}
	As a simple (and somewhat artificial) application, let us rediscover the well-known limit law for the height $H_n$ of a uniform random point on a uniform random Dyck path of size $n$. Using standard methods (e.g. last passage decompositions), the corresponding generating function is found to be
$$
	f(t,s)=\frac{E(t)^2}{1-st E(t)^2}\ \ , \ \ E(t)=1+tE(t)^2.
$$
Using a computer algebra system, it is immediate to eliminate $E(t)$ to find an algebraic equation for $f(t,s)$. We then find automatically a linear recurrence equation for its derivatives at $s=1$. The Maple package gfun~\cite{gfun} gives us:
\begin{align}\label{eq:rayleigh}
	\frac{f^{(k)}(t)}{k!} = -\frac{f^{(k-1)}}{(k-1)!}  +\frac{t}{1-4t}\frac{f^{(k-2)}}{(k-2)!}=0 \ \ , k\geq 2.
\end{align}
	This recurrence (multiplied by $k!$) fits in the framework of the theorem with $L=2$, $a_1=0$, $a_2=\frac{1}{4}k(k-1)$, and $\beta=1/2$. It is then easy to check explicitely the initial conditions, which hold with $\alpha=-\frac{1}{2}$, thus $\ell_0=1$, and with $c_0=2$ and $c_1=1$. The recurrence~\eqref{eq:ck} becomes $c_k=\frac{1}{4} k (k-1) c_{k-2}$,
	so that $c_k=2^{1-k}k!$, and~\eqref{eq:trickmoment} immediately gives:
	$$\frac{\mathbf{E}[(H_n)^k]}{n^{k/2}}\rightarrow\frac{2^{-k}k!\Gamma(\frac{1}{2})}{\Gamma(\frac{k}{2}+\frac{1}{2})}=\Gamma\left(\frac{k}{2}+1\right).
$$
	This implies that $H_n/\sqrt{n}$ converges in distribution to a random variable, which we recognize from its moments as a Rayleigh law of parameter $\sigma=1/\sqrt{2}$ (with density $\frac{x}{\sigma^2}e^{-x^2/(2\sigma^2)}$ on $(0,\infty)$).

	In passing, we remark that we can also apply the theorem with a value of $\beta$ larger than $1/2$. In this case get $a_d\equiv 0$, and $c_k=0$ for $k>0$. We conclude that all positive of moments of $H_n/n^{\beta}$ go to zero -- which is true but weaker than the statement we obtained with the optimal value $\beta=1/2$.
\end{example}

To conclude this section, let us recall that using D-finiteness to compute the coefficients of algebraic functions in the univariate case is a well-known trick, that goes back at least to Comtet~\cite{Comtet} and is extensively used in the context of computer algebra. Here we are only recycling this idea in the context of  bivariate asymptotics.

\subsection{Comments, questions, and plan of the paper}

When we brought Theorem~\ref{thm:tamari} to the attention of Nicolas Curien, he mentioned a book in preparation (joint with Jean Bertoin and Armand Riera) containing (among other things) general convergence results for a wide class of decomposition trees including the peeling trees of random planar maps and critical parking trees. These results will apply as well to Tamari intervals and describe the scaling limit of the (whole) upper path $Q$ in terms of a certain conditionned self similar Markov tree which they introduce. 
It is not clear however that these methods lead to the explicit form of the limit law for the typical height obtained in this paper.
In the other direction, we suspect that, all calculations made, the methods of our paper should enable to obtain the limit law of the height of a uniform random point in other models in this class, for example critical random Cayley parking trees (but have not tried to do so). Note that the limiting first  moment of the height of a random point for Cayley parking trees was already computed from generating functions in~\cite{ContatCurien}.

\medskip
A related question is to ask for the universality of the random variable $Z^{1/4}$ appearing in our asymptotic theorems. It is now well understood that bivariate generating functions $F(t,x)$ satisfying non-linear polynomial equations with one catalytic variable give rise to asymptotic counting exponents of the form $n^{-5/2}$, under natural positivity hypotheses for the equation (see~\cite{DrmotaNoyYu}). To each such equation, one can naturally associate a decomposition tree, that roughly speaking tracks the recursive calculation of coefficients of the function $F(t,x)$. It is natural to expect that the random variable $Z^{1/4}$ and the scaling exponent $n^{3/4}$ are universal in this context, possibly with slightly more restrictive hypotheses. We leave this question open -- studying the behaviour of factorial moments by induction from the equation as we do here could be a way to try to prove this, but combining the estimates of~\cite{DrmotaNoyYu} with the aforementioned probabilistic methods could be more direct.

\medskip

The paper is organized as follows.
In Section~\ref{sec:tamari} we solve the exact counting problem which underlies Theorem~\ref{thm:tamari} by studying the classical  equation with one catalytic variable for Tamari intervals, which we enrich by an extra variable marking the height of a marked point. See Theorem~\ref{thm:exactUp}.

In Section~\ref{sec:tamariLow} we solve the exact counting problem which underlies Theorem~\ref{thm:tamariLow} which deals with the height of the lower path. See Theorem~\ref{thm:exactDown}. Our proof is more technical, as we only manage to write an equation with \emph{two} catalytic variables for this problem (enriched, once again, by an extra variable for the height). It turns out that this equation with two catalytic variables is not too hard to solve, as it can be viewed as two nested equations with one catalytic variable each, which can be solved by standard methods.

In Section~\ref{sec:jointUpSteps} we solve the exact counting problem which underlies Theorem~\ref{thm:mixed} in which we manage to compare the heights of steps on both paths. See Theorem~\ref{thm:exactMixed}. The method is similar to the one used in Section~\ref{sec:tamariLow}.

In Section~\ref{sec:asymptotics}, we prove Theorems~\ref{thm:tamari} and~\ref{thm:tamariLow} by performing the asymptotics of moments from the algebraic equations obtained in Theorems~\ref{thm:exactUp} and~\ref{thm:exactDown}. The proofs are  immediate verifications given Theorem~\ref{thm:trick} and computer algebra calculations. We also prove Theorem~\ref{thm:mixed} which is a direct consequence of Theorem~\ref{thm:exactMixed}.


\smallskip

We will try in this paper to focus on the methods, and leave the actual calculations to the accompanying Maple worksheet~\cite{maple}.
Whether these calculations could be done or verified by hand will not be our concern.

\section{Upper path: exact solution}
\label{sec:tamari}

\subsection{The classical equation and its solution, after \cite{Chapoton, BMFPR}}
\label{sec:decomp}

\begin{figure}
	\centering
	\includegraphics[width=\linewidth]{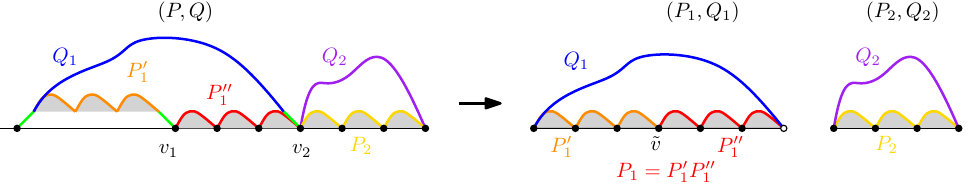}
	\caption{The classical decomposition of Tamari intervals. To the left, an interval of size $n+1$, where $v_1,v_2$ are the first contacts of the lower and upper path, respectively. The decomposition gives rises, to the right, to two Tamari intervals of total size $n$, the first of which has a marked contact, called here $\tilde{v}$. This construction is bijective.}\label{fig:decomp1}
\end{figure}

Let $(P,Q)$ be a Tamari interval. Following~\cite{BMFPR}, we call \emph{contact} of a path a vertex of that path lying on the $x$-axis. We now present a recursive decomposition of Tamari intervals based on contacts of the lower path, following~\cite{Chapoton, BMFPR}.
See Figure~\ref{fig:decomp1} as a support for this discussion.


Let $n\geq 0$ and let $(P,Q)\in \mathcal{I}_{n+1}$ be a Tamari interval. Let $v_1$ and $v_2$ be respectively the leftmost contact of $P$ and $Q$ different from the origin. Then $v_2$ is also a contact of $P$ and it is preceded by a down-step in both $P$ and $Q$. We can thus split paths as $P=uP_0dP_2$, $Q=uQ_1dQ_2$, where we write paths as words with letters $u/d$ representing up/down steps respectively, and where $P_2,Q_2$ start at the contact $v_2$. Moreover, $v_1$ is a contact of $uP_0d$ (possibly equal to $v_2$), so we can write $uP_0d=uP'_1dP''_1$, with $P''_1$ possibly empty, and $P_1''$ starting at $v_1$.
We now define the path $P_1:=P_1'P_1''$, which is naturally equipped with a marked contact~$\tilde{v}$. See Figure~\ref{fig:decomp1} again.

It is proved in~\cite{Chapoton, BMFPR} that $(P_1,Q_1)$ and $(P_2,Q_2)$ are two Tamari intervals. Moreover, this operation is a \emph{bijection} (!) between intervals of size $n+1$ and pairs of intervals of total size $n$, such that the first interval of the pair has a distinguished contact on its lower path.
To translate this recursive construction into an equation for enumeration, one needs to introduce a two-parameter generating function of Tamari intervals:
$$
F(x)\equiv F(t,x) := \sum_{n\geq 0} t^n \sum_{(P,Q)\in \mathcal{I}_n} x^{\mathrm{contact}(P)}
$$
where $\mathrm{contact}(P)$ is the number of contacts of the lower path $P$.

Note that, if a path $\tilde{P}_1$ has $k$ contacts, there are $k$ possible ways to mark a contact in $\tilde{P}_1$, thus decomposing it as $\tilde{P}_1=\tilde{P}_1'\tilde{P}_1''$. When going through these $k$ choices, the number of \emph{non-final} contacts of the path $\tilde{P}_0:=u\tilde{P}_1'd\tilde{P}''_1$ goes through the values $1, 2, \dots, k$, see Figure~\ref{fig:operator}.
Therefore, at the level of generating functions, the operation of marking a contact in an interval $(\tilde{P}_1, \tilde{Q}_1)$ and transforming into an interval of the form $(\tilde{P},\tilde{Q})$ with this construction is taken into account by the operator:
\begin{align}\label{eq:operator}
	&\Delta_x: x^k\hspace{3mm} \longmapsto x^k+\dots +x^{1} = x \frac{x^k-1}{x-1}\\
	&\Delta_x: f(x) \longmapsto x \frac{f(x)-1}{x-1}.
\end{align}
The reason why we choose to count only non-final contacts is to avoid counting the final contact twice when concatenating intervals (see Figure~\ref{fig:decomp1} again).


This observation being made, the recursive decomposition immediately translates into the following functional equation~\cite{Chapoton, BMFPR}
\begin{align}\label{eq:catalyticTamariSimple}
	F(x) = x+xtF(x)\frac{F(x)-F(1)}{x-1}.
\end{align}
Indeed, the first term accounts for the empty path (of size $n=0$), and in the second term the factor $F(x)$ is the contribution of the interval $(P_2,Q_2)$, and the factor $x\frac{F(x)-F(1)}{x-1}=\Delta_x F(x)$ is the contribution of the interval $(P_1, Q_1)$.

\begin{figure}
	\centering
	\includegraphics[width=\linewidth]{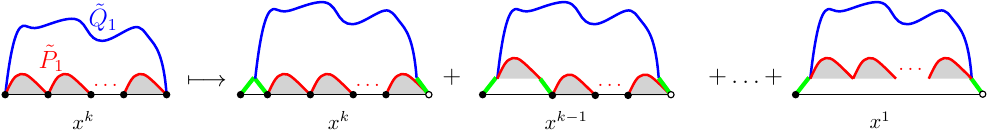}
	\caption{Tracking the number of contacts, and how the divided difference operator appears. On the left, the power of $x$ marks all contacts, while on the right it only marks contacts which are not the last one.}\label{fig:operator}
\end{figure}

Equation~\eqref{eq:catalyticTamariSimple} is a typical example of a \emph{functional equation with one-catalytic variable}, which is the name given in the combinatorics literature\footnote{We have been told that a part of the community is shifting towards the fancier terminology "discrete differential equation".
} to polynomial equations involving a bivariate series $F(t,x)$ as well as its specialisation at $x=1$. Such equations are not easy to solve \emph{a priori}, but fortunately a full theory has been developed in~\cite{BMJ}, which makes their resolution essentially automatic. In the present case, the solution can be written especially nicely with a rational parametrization as follows.
\begin{proposition}[{\cite[Thm. 10]{BMFPR}}]
	The generating functions $F(x)$ and $F(1)$ have the explicit rational parametrisations:
\begin{align}\label{eq:solutionTamariSimple}
	F(x)=\frac{1+u}{(1+zu)(1-z)^3}(1-2z-z^2u)
	\  \  \ , \ \ \ F(1)=\frac{1-2z}{(1-z)^3},
\end{align}
with 
\begin{align}\label{eq:z}
t=z(1-z)^3  , \\ 
	x=\frac{1+u}{(1+zu)^2}. \label{eq:u}
\end{align}
\end{proposition}

\subsection{The enriched equation and its solution}

One can apply the same decomposition as above to Tamari intervals in which the upper path $Q$ carries a marked point: one only has to track recursively the position of the marked point into the different components of the decomposition.
We introduce the generating function
$$
H(x)\equiv H(t,x,s) := \sum_{n\geq 0} t^n \sum_{(P,Q)\in \mathcal{I}_n} x^{\mathrm{contact}(P)} \sum_{i=0}^{2n} s^{Q(i)},
$$
where we recall that $Q(i)$ is the height of the point of the path $Q$ lying at abscissa $i$. It is clear that the height of the marked point can be tracked in the decomposition above. This leads to the functional equation:
\begin{proposition}\label{prop:exactUp}
	The generating function $H(x)\equiv H(t,x,s)$ of Tamari intervals with a marked abscissa where $s$ marks the upper height is solution of the equation:
\begin{align}\label{eq:catalyticTamariUpper}
	H(x) =F(x)+sxtF(x)\frac{H(x)-H(1)}{x-1}+xtH(x)\frac{F(x)-F(1)}{x-1}.
\end{align}
\end{proposition}
\begin{proof}
	This follows directly from the combinatorial decomposition presented in the last section, applied to intervals with a mark abscissa, see Figure~\ref{fig:decomp11}. The first term accounts for the case where the marked abscissa (and height) is zero. The second term accounts for the case where the marked abscissa appears before vertex $v_2$. Through the decomposition the corresponding vertex of the upper path becomes a marked vertex of the path $Q_1$, and its height is shifted by $1$, hence a contribution of $s\Delta_x H(s,x)$ for the interval $(P_1,Q_1)$, while the contribution of the interval $(P_2,Q_2)$, which has no marked abscissa, is just $F(x)$. The third term accounts for the case where the marked abscissa appears at $v_2$ of after, in which case the corresponding vertex becomes a marked vertex of the path $Q_2$, with no shift in the height, hence a contribution of $\Delta_xF(x)$ for $(P_1,Q_1)$ and $H(x,s)$ for $(P_2,Q_2)$.
\end{proof}

\begin{figure}
	\centering
	\includegraphics[width=\linewidth]{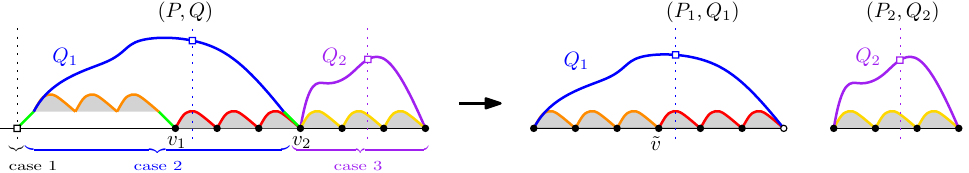}
	\caption{The decomposition of Figure~\ref{fig:decomp1}, now with a marked abscissa. The three cases correspond to the three terms in Equation~\eqref{eq:catalyticTamariUpper}. In cases 2 and 3, the upper point at this abscissa becomes an upper point of the path $Q_1$ and $Q_2$ respectively, and in case 2 the height is shifted by one.}\label{fig:decomp11}
\end{figure}

Our most skeptical readers might convince themselves (again) that we haven't forgotten any terms in the equation by noting that $H(t,x,1)=(2\frac{xd}{dx}+1)F(x)$, since there are $2n+1$ possible abscissas to mark in an interval of size $n$, and that Equation~\eqref{eq:catalyticTamariUpper} for $s=1$ can be obtained by applying the operator $2\frac{xd}{dx}+1$ to~\eqref{eq:catalyticTamariSimple}.

\medskip

Note that, since $F(x)$ and $F(1)$ are known, this equation for $H(x)$ is again a polynomial equation with catalytic variable $x$, in which the variable $s$ only plays the role of a parameter. In some sense, this equation is simpler than the previous one since it is \emph{linear} in $H$. Linear equations with one catalytic variable are easy to solve, at least in principle, via the \emph{kernel method}, which we now apply (see~\cite[p.~508]{FS} for an introduction to the kernel method).

We write~\eqref{eq:catalyticTamariUpper} in kernel form as follows:
\begin{align}\label{eq:kernelForm}
	K(x) H(x) = F(x)-\frac{sxtF(x)H(1)}{(x-1)}
\end{align}
with 
\begin{align}\label{eq:kernel}
	K(x)=\left(-1+sxt\frac{F(x)}{x-1}+xt\frac{F(x)-F(1)}{x-1}\right).
\end{align}
There is a unique formal power series $X_0(t)\in \mathbb{Q}[s][[t]]$ that is such that 
$$
K(X_0)=0.
$$
Indeed, multiplying by $x-1$, this equation has the form 
$X_0(t)=1+t\text{Polynomial}(X_0(t),t,s)$ which enables one to compute the coefficients of $X_0(t)$ recursively. 
We will now apply the kernel method. The most obvious way to do it consists in substituting $x$ for $X_0(t)$ and solving for $H(1)$,
but we find more convenient to work directly with the variables $(u,z)$ rather than $x,t$.
Expressing the kernel in terms of these variables we have (see~\cite{maple} for the calculation)
$$
\tilde{K}\equiv \tilde{K}(u) := K(x)\Big|_{t=z(1-z)^3\atop x=\frac{1+u}{(1-zu)^2}}=\frac{P(u,z,s)}{(uz+1)u(uz^2+2z-1)},
$$
with 
$$
P(u,z,s)=s(1+u)^2z(uz^2+2z-1)+u(1-z)^3.
$$
Since $P(u,z,s)=u-sz + O(z^2+u^2+zu)$, there is a unique series $U\equiv U(z,s) \in \mathbb{Q}[s][[z]]$ such that $P(U,z,s)=0$. Its expansion starts as 
$$U=sz + O(z^2).$$ 
Moreover, since $P$ has degree $1$ in $s$, we can also write $s$ as a rational function of $U$ and $z$,\begin{align}\label{eq:srat}
	s=\frac{U(1-z)^3}{z(1+U)^2(1-Uz^2-2z)}.
\end{align}

We obtain
\begin{theorem}\label{thm:exactUp}
	The series $H(1)\equiv H(t,s,1)$ of Tamari intervals with a marked abscissa, where $t$ marks the size, and $s$ the upper height at this abscissa, has the following rational parametrisation:
	\begin{align}\label{eq:Hsexplicit}
		H(1)= \frac{(1-2z-Uz^2)^2(1+U)}{(1-z)^6}
\end{align}
	with the change of variables $(t,s)\leftrightarrow (z,U)$ given by~\eqref{eq:z} and~\eqref{eq:srat}.
\end{theorem}
\begin{proof}
	We substitute $t=z(1-z)^3$ and $x=\frac{1+u}{(1-zu)^2}$ in~\eqref{eq:kernelForm}, and we further substitute
	$u=U(z,s)$, which cancels the left-hand side. From the right-hand side, and using the explicit expression~\eqref{eq:solutionTamariSimple} of $F(1)$ in terms of $z$  we thus get an expression for $H(1)$ as a rational function of $z$, $s$, and $U=U(z,s)$, in which $s$ can be substituted by~\eqref{eq:srat}, thus giving an expression for $H(1)$ as a rational function of $U$ and $z$. The only thing to be done is to perform the explicit calculation. See~\cite{maple}.
\end{proof}

\section{Lower path: exact solution}
\label{sec:tamariLow}

We will now apply the same decomposition as in the previous sections, but keep track of the height of points on the lower path $P$. In order to do this, we will have to treat differently the contacts of $P$ which appear before or after the marked abscissa, which will force us to work with \emph{two} catalytic variables.
We write $\mathrm{contact}_{< i}(P), \mathrm{contact}_{\geq i}(P)$ for the number of contacts of $P$ strictly before, or weakly after, abscissa $i$, respectively.

\subsection{The enriched equation}

We introduce the generating function
$$
G(x,y)\equiv G(t,x,y,w) := \sum_{n\geq 0} t^n \sum_{(P,Q)\in \mathcal{I}_n}  \sum_{i=0}^{2n}
w^{P(i)}x^{\mathrm{contact}_{< i}(P)}y^{\mathrm{contact}_{\geq i}(P)}.
$$
We have
\begin{proposition}\label{prop:catalyticTamariLower}
	The generating function $G(x,y)\equiv G(t,x,y,w)$ of Tamari intervals with a marked abscissa where $w$ marks the lower height is solution of the equation:
\begin{align}
	G(x,y) &=
	F(y)
+
	txw\frac{G(1,y)-G(1,1)}{y-1}
	F(y)
+
tx \frac{F(y)-yF(1)}{y-1} F(y) \nonumber \\
	&+
	t\frac{x^2}{y} \frac{G(x,y)-\frac{y}{x}F(x)-G(1,y)+yF(1)}{x-1} F(y) 
		+tx\frac{F(x)-F(1)}{x-1} G(x,y).\label{eq:catalyticTamariLower}
\end{align}
\end{proposition}
\begin{figure}
	\centering
	\includegraphics[width=\linewidth]{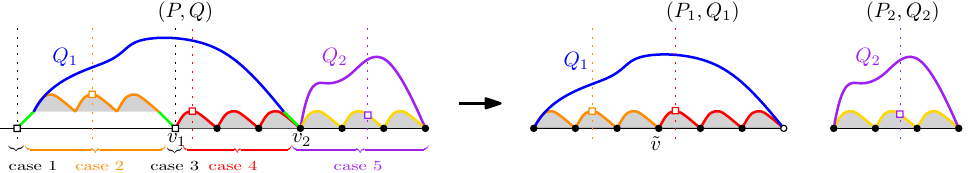}
	\caption{The decomposition of Figure~\ref{fig:decomp1}, with a marked abscissa, but we now track the height of points of the lower path. There are now five cases, which correspond to the five terms in Equation~\eqref{eq:catalyticTamariLower}.}\label{fig:decomp12}
\end{figure}

\begin{proof}
	Given an interval $(P,Q)$ of size $n+1$ with a marked abscissa, we apply again the decomposition of Section~\ref{sec:decomp}. 
	As before we let $v_1$, $v_2$, be the first non-initial contacts of the paths $P$ and $Q$, respectively, and we also let $i_1, i_2$ be their abscissa. We let $i\in [0,2n]$ be the marked abscissa.
	We distinguish five cases (we will treat the third one last), see Figure~\ref{fig:decomp12},
\begin{itemize}
	\item Case 1: $i=0$. Then we are only  considering a Tamari interval, where all contacts are marked by the variable $y$. The corresponding contribution if $F(y)$.
	\item Case 2: $i \in (1,i_1-1)$. In the decomposition, the corresponding vertex of $P$ becomes a marked vertex of $P_1$ with a shift of $1$ in the height, hence a contribution of $w$. 
		Moreover, configurations in this case are obtained by applying the construction  of Figure~\ref{fig:operator} but restricting it to contacts appearing after the marked abscissa, see Figure~\ref{fig:operator2}(up). Therefore, an interval $(\tilde{P_1}, \tilde{Q_1})$ having a marked abscissa, with respectively $k$ and $\ell$ contacts stricly before, and weakly after, this abscissa (thus having a contribution of $x^ky^\ell$ in $G(x,y)$) gives rise to $\ell$ intervals contributing to this case, with a contribution of 
		$$
		x^1 y^{\ell-1} + x^1 y^{\ell-2} + \dots + x^1 y^0 = x \frac{y^\ell-1}{y-1}.
		$$
In total, the contribution for the first interval is thus $x \frac{G(1,y)-G(1,1)}{y-1}$. The contribution of the second interval $(P_2,Q_2)$ is just $F(y)$, since all corresponding contacts appear after the marked abscissa. In total, this cases gives a contribution of 
		$t x w \frac{G(1,y)-G(1,1)}{y-1} F(y)$, which is the second term in~\eqref{eq:catalyticTamariLower}.
	\item  Case 4: $i \in (i_1,i_2)$. In the decomposition, the vertex of $P$ becomes a vertex of $P_1$ with no shift in height. 
		Moreover, these configurations are obtained by applying the construction  of Figure~\ref{fig:operator} but restricting it to contacts appearing before the marked abscissa, see Figure~\ref{fig:operator2}(down). More precisely, an interval $(\tilde{P_1}, \tilde{Q_1})$ having a marked abscissa which is not the final one, with respectively $k$ and $\ell$ contacts stricly before, and weakly after, this abscissa (thus having a contribution of $x^ky^\ell$ in $G(x,y)$) gives rise to $k$ intervals contributing to this case, with a contribution of 
		$$
		x^{k+1} y^{\ell} + x^{k-1} y^{\ell-1} + \dots + x^2 y^{\ell-1} = \frac{x^2}{y}\frac{x^k-1}{x-1} y^\ell.
		$$
In total, the contribution for the first interval is thus 
		$$t\frac{x^2}{y} \frac{\tilde{G}(x,y)-\tilde{G}(1,y)}{x-1},$$
		where $\tilde{G}$ is the same generating function as $G$, but in which the marked abscissa is not the last one.
We have
$$\tilde{G}(x,y) = G(x,y) - \frac{y}{x}F(x),
$$
since an interval with the last abscissa marked is the same as an unmarked interval (and the corresponding vertex has height zero), and since in such a case all contacts are counted with weight $x$ except the last one which receives a weight $y$.

		Moreover, the contribution of the second interval $(P_2,Q_2)$ is just $F(y)$, since all corresponding contacts appear after the marked abscissa. In total, this case gives a contribution of 
		$t\frac{x^2}{y} \frac{\tilde{G}(x,y)-\tilde{G}(1,y)}{x-1} F(y)$, 
		which is the fourth term in~\eqref{eq:catalyticTamariLower}.
	\item  Case 5: $i\geq i_2$. In this case, the marked vertex becomes a marked vertex of $Q_2$ and there is not shift in  height. The contribution of $(P_2,Q_2)$ is simply $G(x,y)$, while the contribution of the first interval $(P_1,Q_1)$ is $xt \frac{F(x)-F(1)}{x-1}$ as in the classical case, since all its contacts appear before the marked abscissa. This gives the last term in~\eqref{eq:catalyticTamariLower}.
	\item Case 3: $i=i_1$ and $i_1\not\in\{0,i_2\}$. Such configurations are made of the concatenation of two intervals.
		The second one has no constraint, hence a contribution of $F(y)$. The first one has several constraints. First, its upper path has only two contacts. Arguing as in the classical decomposition, the generating function for intervals where the upper path has two contacts is 
$$ ty\frac{F(y)-F(1)}{y-1}.
$$
		Second, in the first interval, the first nonzero contact of the lower path should be different from the one of the upper path (since we want $i_1\neq i_2$ in the end). We thus have to substract the corresponding contribution, and in total we find that the contribution of the first interval is 
		$$
		ty\frac{F(y)-F(1)}{y-1} - ty F(1).
		$$
		Putting things together, the  generating function for Case 3 is 
		$$
		\frac{x}{y} \left(ty\frac{F(y)-F(1)}{y-1} -tyF(1)\right) F(y),
		$$
		where the prefactor $\frac{x}{y}$ accounts for the initial contact. 
This is the third term in~\eqref{eq:catalyticTamariLower}.
\end{itemize}
\end{proof}
For skeptics, we have (successfully) tested the expansion of $G(x,y)$ given by Equation~\ref{eq:catalyticTamariLower} up to order $n=8$, using an exhaustive generation of Tamari lattices made by independent means.

\begin{figure}
	\centering
	\includegraphics[width=\linewidth]{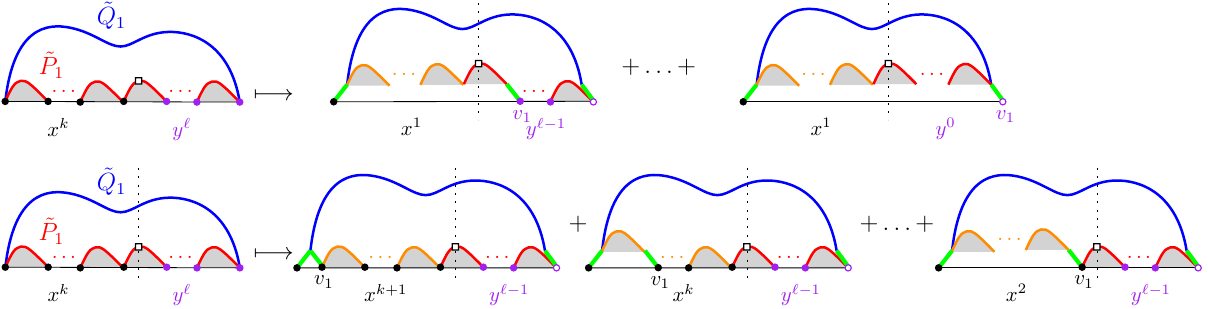}
	\caption{Refining the operator of Figure~\ref{fig:operator} to distinguish the case where the marked abscissa is before (upper figure) or after (lower figure) the contact $v_1$.}\label{fig:operator2}
\end{figure}


\subsection{Solution}
\label{sec:solutionLow}

Although equations with two catalytic variables are notoriously difficult, Equation~\eqref{eq:catalyticTamariLower} is of a very particular kind. First, it is again \emph{linear} in the main unknown $G(x,y)$. Second, the equation involves the specializations $G(x,y)$, $G(1,y)$, and $G(1,1)$, but not $G(x,1)$. As we will see, this will enable us to treat this equation as two nested equations, each having only one catalytic variable\footnote{After we prepared this work, Mireille Bousquet-Mélou has informed us of a paper of her in preparation with Hadrien Notarantonio, studying such  equations with several catalytic variables where some variable specializations are missing. They also solve them iteratively as "nested" equations as we do here, and they prove that this works in a generic context with mild hypotheses. It is not clear that their theorems apply to our case since our equations also involve the known function $F(x)$ and are thus not directly polynomial, but we definitely recommend having a look at that paper when it is out!}.
The reader interested in more difficult cases of equations with two catalytic variables might consult, for example, the reference~\cite{BBMR} for entry points into this fastly growing literature.

\subsubsection{\noindent \bf First step: eliminating variable $x$ (or $u$).}

We start by writing Equation~\eqref{eq:catalyticTamariLower} in kernel form
\begin{align}\label{eq:kernelForm2}
	K_2(x,y) G(x,y)=&
	F(y)
	+
	txw\frac{G(1,y)-G(1,1)}{y-1}
	F(y)\nonumber \\
	&+
tx \frac{F(y)-yF(1)}{y-1} F(y) 
	+
	t\frac{x^2}{y} \frac{-\frac{y}{x}F(x)-G(1,y)+yF(1)}{x-1} F(y) .
\end{align}
with
\begin{align}\label{eq:kernel2}
	K_2(x,y) =\frac{tx^2F(y)}{y(x-1)}-1+xt\frac{F(x)-F(1)}{x-1}.
\end{align}

We will now work under the change of variables~\eqref{eq:z}, \eqref{eq:u}. Since we have two catalytic variables, we introduce a new variable $v$, which is to $y$ what $u$ is to $x$. To summarize, we write
\begin{align}\label{eq:zuv}
t=z(1-z)^3  ,  
	x=\frac{1+u}{(1+zu)^2},
	y=\frac{1+v}{(1+zv)^2},
\end{align}
We write respectively $\tilde{G}(u,v), \tilde{G}_1(v), \tilde{G}_{11}, \tilde{K}_2$ for the quantities $G(x,y),G(1,y),G(1,1),K_2$ expressed in the variables $z,u,v$ after the substitutions~\eqref{eq:zuv}.

By~\eqref{eq:solutionTamariSimple}, the kernel $\tilde{K}_2$ is an explicit rational function of $z,u,v$.  Its numerator is a polynomial $\tilde{k}_2(z,u,v,w)$ which, computations made~\cite{maple}, is given by 
\begin{align*}
	\tilde{k}_2(u,v)=
	((vz^2+z^2)u+vz^2+2z-1)((vz^2-z^2+3z-1)u+vz^2+z).
\end{align*}
The equation $\tilde{k}_2=0$ has a unique root $U_0\equiv U_0(z,v)$ which is a power series in $z$, given by the root of the second factor. Explicitly, it is given by with $U_0=\frac{z(1+vz)}{1-3z+z^2-vz^2}=z+O(z^2)$.

Making the substitutions~\eqref{eq:zuv} and using the known expressions of $F(x)$, $F(y)$, Equation~\eqref{eq:kernelForm2} takes the form
$$
\tilde K_2 (u,v) \tilde{G}(u,v) = \tilde{L}_2(z,u,v,\tilde{G}_1(v),\tilde G_{11})
$$
for some rational function $\tilde{L}_2$ that can be written explicitly (its numerator has 228 terms, \cite{maple}).
Substituting $u=U_0$ in this equation, we cancel the left-hand side,
hence we also cancel the right-hand side.
We are thus left with the following polynomial equation satisfied by 
$z,v,w, \tilde{G}_1(v)$ and $\tilde{G}_{11}$,
\begin{align}\label{eq:B}
\tilde{L}_2(z,U_0(z,v),v,\tilde{G}_1(v),\tilde G_{11})=0.
\end{align}
The numerator of this equation has 91 terms, but it has only degree one in $\tilde{G}_1(v)$.
At this stage, we have eliminated the unknown $\tilde{G}(u,v)$ and the  variable $u$.

\subsubsection{\noindent \bf Second step: eliminating variable $y$ (or $v$).}
\label{sec:eliminating2}

At this stage, we have shown that the series
$\tilde{G}_1(v)$ and $\tilde{G}_{11}$ satisfy the explicit polynomial equation~\eqref{eq:B}.
But we recognize again this equation as an equation with one catalytic variable, which is now the variable $v$!
Even more, since the equation is linear in $\tilde{G}_1(v)$, we can just use the kernel method again. We write the numerator of~\eqref{eq:B} in the form
$$
a(\tilde{G}_{11},v,z,w)\tilde{G}_1(v)+
b(\tilde{G}_{11},v,z,w)=0,
$$
%
where in fact $a(\tilde{G}_{11},v,z,w)=a(v,z,w)$ does not depend on $\tilde{G}_{11}$. A direct check shows that the equation $a(V_0,z,w)=0$ as a unique solution $V_0(z)$ which is a formal power series in $z$, with $V_0=zw+O(z^2)$.
Substituting $v=V_0$ in the equation, we finally obtain the equation
$$
b(\tilde{G}_{11},V_0(z),z,w)=0,
$$
which shows that $\tilde{G}_{11}$ is algebraic.
Explicitly, we can eliminate the variable $v$ from the equations $a(v,z,w)$ and $b(\tilde{G}_{11},v,z,w)$, and we obtain a polynomial equation for the function $\tilde{G}_{11}$.
It turns out that this equation is not even so big, and computations made~\cite{maple} we finally obtain:
\begin{theorem}\label{thm:exactDown}
	The generating function $G(x,y)$ is an algebraic function. Moreover, the function $G(1,1)$ after the change of variables $z\leftrightarrow t$ given by~\eqref{eq:z} satisfies the polynomial equation $C(G(1,1),z,w)=0$ with
	\begin{align}\label{eq:exactDown}
		C(h,z,w)&= 
wz(-1+z)^9h^3+(-1+z)^6(2w^2z^2-w^2z+2z^2+w-z)h^2\nonumber \\
		&	-(-1+z)^3(w^2z^3-3w^2z^2-2wz^3+w^2z-2wz^2+z^3+5wz-3z^2-2w+z)h\nonumber \\
		&	  +4wz^2-4wz+w.
	\end{align}
\end{theorem}

\section{Tracking marked up steps on both paths}
\label{sec:jointUpSteps}

We will now apply the same decomposition as in the previous sections, but consider paths with a marked up step, rather than a marked abscissa. As mentioned in the introduction, the reason to do that is that in this setting we are able to study the joint law of the heights of the two paths. Our goal is to prove Theorem~\ref{thm:mixed}.

Recall that we let $\tilde P(j)$ be the height (of the initial point of) the $j$-th up step of the path $P$.

We write $\mathrm{contact'}_{\leq j}(P), \mathrm{contact'}_{\geq j}(P)$ for the number of contacts of $P$ weakly before, or after, the $j$-th up step of $P$, respectively.

\subsection{The enriched equation}

We introduce the generating function
$$
J(x,y)\equiv J(t,x,y,r,s) := \sum_{n\geq 0} t^n \sum_{(P,Q)\in \mathcal{I}_n}  \sum_{j=1}^{n}
r^{\tilde{P}(j)}s^{\tilde{Q}(j)} 
x^{\mathrm{contact'}_{\leq j}(P)}y^{\mathrm{contact'}_{\geq j}(P)}.
$$
We have
\begin{proposition}\label{prop:catalyticTamariMixed}
	The generating function $J(x,y)\equiv J(t,x,y,r,s)$ of Tamari intervals where the $j$-th up step of each path is marked for some integer $j$, where $r$ marks the lower height and $s$ the upper height is solution of the equation:
\begin{align}
	J(x,y) &=
	x\frac{F(y)-y}{y} 
	+  rstx\frac{J(1,y)-J(1,1)}{y-1} F(y) 
	+  stx^2\frac{J(x,y)-J(1,y)}{(x-1)y}F(y)
	+xtJ(x,y)\frac{F(x)-F(1)}{x-1}.
\label{eq:catalyticTamariMixed}
\end{align}
\end{proposition}
\begin{figure}
	\centering
	\includegraphics[width=\linewidth]{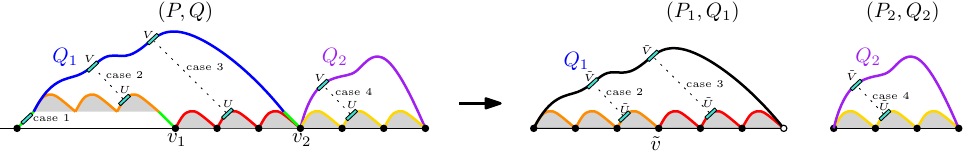}
	\caption{The decomposition of Figure~\ref{fig:decomp1}, where the $j$-th up step for some $j\in[n]$ is marked on both paths. There are four cases to consider,  corresponding to the four terms in Equation~\eqref{eq:catalyticTamariMixed}.}\label{fig:decomp13}
\end{figure}

\begin{proof}
	The proof is very similar to the proof of Proposition~\ref{prop:catalyticTamariLower}, except that we keep track of the height on both paths. It is actually less involved since there are less cases to consider.
	Given an interval $(P,Q)$ of size $n+1$ with a marked $j\in [n]$, we apply again the decomposition of Section~\ref{sec:decomp}.
We let $U, V$ be the $j$-th up step of $P$ and $Q$, respectively.
	As before we let $v_1$, $v_2$, be the first non-initial contacts of the paths $P$ and $Q$.
	We distinguish four cases, see Figure~\ref{fig:decomp13},
\begin{itemize}
	\item Case 1: $U=V$ is the first step of the path ($j=1$).
		In this case we are only counting intervals by contacts as in the series $F(y)$. However we have to exclude the empty interval (hence a correcting term $-y$), and we should count all contacts with weight $y$ except for the first one, hence a correcting factor $x/y$. This gives the first term in~\eqref{eq:catalyticTamariMixed}.
	\item Case 2: $U$ is not the initial step of $P$, but $U$ appears before $v_1$. In this case $V$ belongs to $Q_1$. In this case, after applying the decomposition, $U,V$ give rise naturally to marked contacts $\tilde{U},\tilde{V}$ of $P_1,Q_1$ appearing before the vertex $\tilde v$. The height of both steps is shifted by $1$, hence the factor $rs$. The remaining factors are understood in a similar way as in case 2 of the proof of Proposition~\ref{eq:catalyticTamariLower}. Namely, an interval $(\tilde{P_1}, \tilde{Q_1})$ with marked up steps $\tilde{U}, \tilde{V}$, with respectively $k$ and $\ell$ contacts  before and after $\tilde{U}$ (thus having a contribution of $x^ky^\ell$ in $J(x,y)$) gives rise to $\ell$ intervals contributing to this case, with a contribution of 
		$$
		x^1 y^{\ell-1} + x^1 y^{\ell-2} + \dots + x^1 y^0 = x \frac{y^\ell-1}{y-1}.
		$$
In total, the contribution for the first interval is thus $x \frac{J(1,y)-J(1,1)}{y-1}$. The contribution of the second interval $(P_2,Q_2)$ is just $F(y)$, since all corresponding contacts appear after the marked abscissa. In total, we obtain the second term in~\eqref{eq:catalyticTamariMixed}.

	\item Case 3: $U$ appears (weakly) after $v_1$ but before $v_2$. In this case again $V$ belongs to $Q_1$. In this case, after applying the decomposition, $U,V$ give rise naturally to marked contacts of of $P_1,Q_1$ appearing (weakly) after the vertex $\tilde v$. Only the height of the steps of the upper path is shifted by $1$, hence the factor $s$. The rest of factors are understood in a similar way as in case 4 of the proof of Proposition~\ref{prop:catalyticTamariLower}, but is simpler.
		Namely,   an interval $(\tilde{P_1}, \tilde{Q_1})$ with marked up steps $\tilde{U},\tilde{V}$, with respectively $k$ and $\ell$ contacts  before and after $\tilde{U}$ (thus having a contribution of $x^ky^\ell$ in $G(x,y)$) gives rise to $k$ intervals contributing to this case, with a contribution of 
		$$
		x^{k+1} y^{\ell} + x^{k-1} y^{\ell-1} + \dots + x^2 y^{\ell-1} = \frac{x^2}{y}\frac{x^k-1}{x-1} y^\ell.
		$$
In total, the contribution for the first interval is thus 
		$$t\frac{x^2}{y} \frac{J(x,y)-J(1,y)}{x-1}.$$
		(Note that the situation is simpler than in Proposition~\ref{prop:catalyticTamariLower} since we do  not have to deal with the subtle case where the vertex whose height is considered is the last on the path -- since we mark only \emph{up} steps).
			Moreover, the contribution of the second interval $(P_2,Q_2)$ is just $F(y)$, and in total we obtain the third term in~\eqref{eq:catalyticTamariMixed}.

	\item Case 4: $U$ appears (weakly) after $v_2$. In this case $V$ belongs to $Q_2$. This case is similar to the last case of the proof of Proposition~\ref{eq:catalyticTamariLower} and directly gives the fourth term in~\eqref{eq:catalyticTamariMixed}
\end{itemize}
\end{proof}

\subsection{The specialized equation}
In view of proving Theorem~\ref{thm:mixed}, we will consider the 
following specialisation of the function $J(x,y)$:
$$
M(x,y)\equiv M(t,x,y,w) := J(x,y)\Big|_{s=w,r=w^{-3}}
=
\sum_{n\geq 0} t^n \sum_{(P,Q)\in \mathcal{I}_n}  \sum_{j=1}^{n}
w^{\tilde{Q}(j)-3\tilde{P}(j)} 
x^{\mathrm{contact'}_{\leq j}(P)}y^{\mathrm{contact'}_{\geq j}(P)}.
$$
From Proposition~\ref{prop:catalyticTamariMixed} it satisfies the equation
\begin{align}
	M(x,y) &=
	x\frac{F(y)-y}{y} 
	+  w^{-2} tx\frac{M(1,y)-M(1,1)}{y-1} F(y) 
	+  wtx^2F(y)\frac{M(x,y)-M(1,y)}{(x-1)y}
	+xt\frac{F(x)-F(1)}{x-1}M(x,y).
\label{eq:catalyticTamariMixedM}
\end{align}
We will now solve this equation.

\subsection{Solving the equation}

Our method to solve~\eqref{eq:catalyticTamariMixedM} is similar to the one used in Section~\ref{sec:solutionLow}, except that the equation to be solved in the second step will now be quadratic instead of linear.

\subsubsection{\noindent \bf First step: eliminating variable $x$ (or $u$).}

We start by writing Equation~\eqref{eq:catalyticTamariMixedM} in kernel form
\begin{align}\label{eq:kernelForm2M}
	\mathbf{K}_2(x,y) M(x,y)=&
	x\frac{F(y)-y}{y} 
	+  w^{-2} tx\frac{M(1,y)-M(1,1)}{y-1} F(y) 
	+  wtx^2F(y)\frac{-M(1,y)}{(x-1)y}
\end{align}
with
\begin{align}\label{eq:kernel2M}
	\mathbf{K}_2(x,y) = -1 + 
	+  wtx^2F(y)\frac{1}{(x-1)y}
	+xt\frac{F(x)-F(1)}{x-1}.
\end{align}

We will again work under the change of variables~\eqref{eq:zuv}.
We write respectively $\tilde{M}(u,v), \tilde{M}_1(v), \tilde{M}_{11}, \tilde{\mathbf{K}}_2$ for the quantities $M(x,y),M(1,y),M(1,1),\mathbf{K}_2$ expressed in the variables $z,u,v$ after the substitutions~\eqref{eq:zuv}.

By~\eqref{eq:solutionTamariSimple}, the kernel $\tilde{\mathbf{K}}_2$ is an explicit rational function of $z,u,v$.  Its numerator is a polynomial $\tilde{\mathbf{k}}_2(u,v)\equiv\tilde{\mathbf{k}}_2(z,u,v,w)$ which, computations made~\cite{maple}, is given by 
\begin{align*}
	\tilde{\mathbf{k}}_2(u,v)=
u^2v^2wz^4+2uv^2wz^4+3u^2vwz^3+v^2wz^4-u^2vwz^2-u^2z^4+6uvwz^3+2u^2wz^2+3u^2z^3
	\\-2uvwz^2+3vwz^3-u^2wz-3u^2z^2+4uwz^2-uz^3-vwz^2+u^2z-2uwz+3uz^2+2wz^2-3uz-wz+u.
\end{align*}
The equation $\tilde{\mathbf{k}}_2=0$ has a unique root $\mathbf{U}_0\equiv \mathbf{U}_0(z,v)$ which is a power series in $z$. It satisfies $\mathbf{U}_0=wz+O(z^2)$.

Making the substitutions~\eqref{eq:zuv} and using the known expressions of $F(x)$, $F(y)$, Equation~\eqref{eq:kernelForm2M} takes the form
$$
\tilde{\mathbf{K}}_2 (u,v) \tilde{M}(u,v) = \tilde{\mathbf{L}}_2(z,u,v,\tilde{M}_1(v),\tilde M_{11})
$$
for some rational function $\tilde{\mathbf{L}}_2$ that can be written explicitly (its numerator has 99 terms, \cite{maple}).
Substituting $u=\mathbf{U}_0$ in this equation, we cancel the left-hand side,
hence we also cancel the right-hand side.
We are thus left with the following polynomial equation satisfied by 
$z,v,w, \tilde{M}_1(v)$ and $\tilde{M}_{11}$,
\begin{align}
	\tilde{\mathbf{L}}_2(z,\mathbf{U}_0(z,v),v,\tilde{M}_1(v),\tilde M_{11})=0.
\end{align}
We can eliminate $\mathbf{U}_0$ between this equation and $\tilde{\mathbf{k}}_2 (\mathbf{U}_0(z,v),v)$, and we thus obtain~\cite{maple} a polynomial equation of the form 
\begin{align}\label{eq:BM}
	B(z,v,w, \tilde{M}_1(v),\tilde{M}_{11})=0.
\end{align}
This polynomial has 854 terms, and it has degree two in $\tilde{M}_1(v)$.
At this stage, we have eliminated the unknown $\tilde{M}(u,v)$ and the  variable $u$.

\subsubsection{\noindent \bf Second step: eliminating variable $y$ (or $v$).}

Let us write polynomial equation~\eqref{eq:BM} in the form.
$$
\mathbf{a}(\tilde{M}_{11},v,z,w)\tilde{M}_1(v)^2+
\mathbf{b}(\tilde{M}_{11},v,z,w)\tilde{M}_1(v)+
\mathbf{c}(\tilde{M}_{11},v,z,w)=0.
$$
We will solve this equation using (a variant of) the quadratic method (see \cite{BMJ} and references therein).
Introduce notation:
\begin{align}\label{eq:quadraticMethodBM}
	B(v):=	 \mathbf{a}(v) \tilde{M}_1(v)^2 
	+\mathbf{b}(v) \tilde{M}_1(v) 
	+\mathbf{c}(v) \ \ , \ \ 
	B_2(v):= 2 \mathbf{a}(v) \tilde{M}_1(v) 
	+\mathbf{b}(v),
\end{align}
where in the notation we only indicate dependencies in the variable $v$, for example $\mathbf{a}(v)=\mathbf{a}(\tilde{M}_{11},v,z,w)$.
We first claim that there exists a formal power series $V_0(z) = zw^{-1} + O(z^2)$ such that $B_2(V_0(z))=0$. To see this, we generate the first few terms of the series $\tilde{M}_1(v)$
to see~\cite{maple} that
$$
B_2(v) = a_{2,0} v^2 + a_{1,1} vz  + a_{1,2} v z^2  + a_{2,1} v^2 z + a_{3,0} v^3
+O_4(z,v)
$$
where the $a_{i,j}$ are explicit  polynomials in $w$  and where $O_4(z,v)$ is a formal power series in $z,v$ (and coefficients polynomial in $w$) with all terms of homogeneous degree at least 4.
Therefore the equation $B_2(V_0)=0$ can be written 
$$
V_0 = -(a_{2,0})^{-1} ( a_{1,1} z  + a_{1,2} z^2  + a_{2,1} V_0 z + a_{3,0} V_0^2 +V_0^{-1} O_4(z,V_0)).
$$
This is enough to see that the equation $B_2(V_0)=0$ has a formal power series root $V_0(z)$ of the form
$V_0(z) = \frac{-a_{1,1}}{a_{2,0}} z +O(z^2)$. Indeed all coefficients of $V_0$ can be computed inductively from this germ. This proves the claim (note that this is essentially a Newton-Puiseux argument).  Making constants explicit, we have $V_0(z) = zw^{-2} +O(z^2)$. There is also another power series root $V_0$ with constant coefficient $-1$ but we will not use it.

We can now apply (a variant of) the quadratic method. We have two equations involving the series $V_0$: the one that defines it, and the one obtained by substituting $v=V_0$ in~\eqref{eq:BM}. Namely:
$$
B(z,V_0,w, \tilde{M}_1(V_0),\tilde{M}_{11})=0 \ \ ,\ \  B_2(z,V_0,w, \tilde{M}_1(V_0),\tilde{M}_{11})=0.
$$
We can eliminate the series $\tilde{M}_1(V_0)$ between these two equations (which amounts to saying that the discrinant $\mathbf{b}^2-4\mathbf{ac}$ vanishes), and we obtain an equation of the form
$$
P(z,V_0,w,\tilde{M}_{11})=0.
$$
Computations made~\cite{maple}, this equation has three nontrivial irreducible factors, $P_1,P_2,P_3$, where only $P_3$ depends (in fact, linearly) on $\tilde M_{11}$. It is a direct check than none of the $v$-roots of $P_1$ and $P_2$ is equal to $V_0$ (in fact none is a formal power series in $z$), from which we conclude that $P_3(z,V_0,w,\tilde{M}_{11})=0$. This equation can be written as a rational expression of $\tilde M_{11}$ in the variable $V_0$, namely~\cite{maple} we have $\tilde{M}_{11}=\mathbf{R}(z,V_0)$ with
\begin{align}\label{eq:paramM11}
	R(z,V)= \frac{(Vz^2-z^2+3z-1)Vw^2}{Vz^4-3Vz^3+3Vz^2+z^3-Vz-3z^2+3z-1}.
\end{align}

One can check that substituting $v=V_0$ and this expression of $\tilde{M}_{11}$ cancels both $\mathbf{b}$ and $\mathbf{c}$. Therefore it also cancels $\mathbf{a}$, which gives us an explicit equation satisfied by $V_0$, namely~\cite{maple} we have $S(V_0,z)=0$ with:
\begin{align}\label{eq:V0}
	\footnotesize 
		&S(V,z,w)=V^5w^5z^5+V^5w^4z^5+V^5w^3z^5+5V^4w^5z^4
	+V^5w^2z^5-2V^4w^5z^3+5V^4w^4z^4-V^3w^4z^5+V^5wz^5\nonumber\\&-2V^4w^4z^3+5V^4w^3z^4+2V^4w^2z^5+8V^3w^5z^3+3V^3w^4z^4-V^3w^3z^5-2V^4w^3z^3+4V^4w^2z^4+2V^4wz^5\nonumber\\&-6V^3w^5z^2+5V^3w^4z^3+3V^3w^3z^4+V^3w^2z^5-2V^2w^4z^4-V^4w^2z^3+4V^4wz^4+V^3w^5z-5V^3w^4z^2\nonumber\\&+5V^3w^3z^3+8V^3w^2z^4+V^3wz^5+4V^2w^5z^2+7V^2w^4z^3-2V^2w^3z^4-V^4wz^3+V^3w^4z-5V^3w^3z^2+\nonumber\\&V^3w^2z^3+8V^3wz^4-4V^2w^5z-5V^2w^4z^2+7V^2w^3z^3+4V^2w^2z^4+V^3w^3z-V^3w^2z^2+V^3wz^3+V^2w^5\nonumber\\&+V^2w^4z-5V^2w^3z^2+6V^2w^2z^3+4V^2wz^4-V^3wz^2-2V^3z^3+V^2w^3z-5V^2w^2z^2+5V^2wz^3+4Vw^2z^3\nonumber\\&+V^3z^2+3V^2w^2z-2V^2wz^2-4V^2z^3-4Vw^2z^2+3Vwz^3-V^2w^2+3Vw^2z-Vwz^2-2Vz^3+V^2z-Vw^2\nonumber\\&-3Vz^2+2Vz-2z^2+z.
\end{align}

We have thus proved:
\begin{theorem}\label{thm:exactMixed}
	The generating function $M(1,1)\equiv M(1,1;t,w)$ is an algebraic function given by
	$$
	M(1,1) = R(z,V_0)
	$$
	where the change of variables $(t,w)\leftrightarrow (z,V)$ is given by~\eqref{eq:z} and by the equation $S(V_0,z,w)=0$.
\end{theorem}
By eliminating $V_0$ in \eqref{eq:paramM11}-\eqref{eq:V0} one obtains an explicit polynomial cancelling $M(1,1)$. Written in the variables $z$ and $w$, it has $601$ terms and degree five in $M(1,1)$. See~\cite{maple}.
Given the previous calculations it is not difficult to see that the full series $M(x,y)$ is algebraic but we will never use it.

We leave to our readers the task of proving algebraicity of the full function $J(1,1)$ (with both variables $r$ and $s$) rather than of its specialization $M(1,1)$ (involving only $w$) as we do here. We suspect that this is possible with similar methods, but by lack of application in mind we have not tried to do so -- therefore we  do not know if the calculations are manageable nor if unexpected dificulties arise.

\section{Asymptotics of moments}
\label{sec:asymptotics}

In this section we prove Theorems~\ref{thm:tamari} and~\ref{thm:tamariLow}. In both cases this will be a direct application of the method of Section~\ref{sec:trick}, in particular Theorem~\ref{thm:trick}, up to computer algebra calculations done in~\cite{maple}.
We will also prove Theorem~\ref{thm:mixed} which is a direct check.

\subsection{Proof of Theorem~\ref{thm:tamari}}

We start with Equation~\eqref{eq:catalyticTamariUpper} in Theorem~\ref{thm:exactUp}. This equation shows that $f(t)=H(t,1,s)\equiv H(1)$ is algebraic, therefore $H(x,r+1)$ is algebraic. Therefore (see e.g.~\cite{Comtet} again), there exists a recurrence formula of the form~\eqref{eq:Dfinite2} to compute its derivatives
$$
f^{(k)}\equiv f^{(k)}(t) :=\left.\left(\left(\frac{d}{ds}\right)^k H(1)\right)\right|_{s=1}.
$$
Given the form of our parametrization, we will prefer to work under the variable $z$ given by~\eqref{eq:z}, and in fact it is even more convenient to introduce the variable $\delta=(1-4z)^{1/2}$.
It is easily cheched~\cite{maple} that the main singularity of $H(1)$ is at $t=\rho:=\frac{27}{256}$, which corresponds to $z=\frac{1}{4}$, i.e. $\delta=0$. In what follows we will still write $f^{(k)}$ for the function $f^{(k)}(t)$ expressed in the variable $\delta$.

Using the package gfun, we immediately find~\cite{maple} that the $f^{(k)}$ satisfy an equation of the form~\eqref{eq:Dfinite2} with $L=6$,
namely 
\begin{align*}
	f^{(k)}(t)= \sum_{d=1}^{L} h_d(t,k)  f^{(k-d)}(t) 
\ \ , \ \  
k \geq L,
\end{align*}
where the $h_d(t,k)$ for $d=1..6$ are Laurent polynomials in $\delta$ with coefficients which are rational functions of $k$.
The degree of $h_1,\dots,h_6$ in $\delta^{-1}$ are respectively $0,6,6,10,10,10$.
Since $\delta$ behaves as $(1-t/\rho)^{1/4}$, this implies that hypothesis (i) of Theorem~\ref{thm:trick} holds with $\beta=\frac{3}{4}$.
One can explicitly check the values of the corresponding constants $a_d(k)$, which are nothing but top coefficients of $h_d$ in $\delta^{-1}$ up to a scaling factor. They are given~\cite{maple} by
\begin{align}\label{eq:explicitAi}
	a_1(k),\dots,a_6(k) = 0, \frac{\sqrt{6}}{96}(3k-4)(3k-8), 0, 0, 0, 0.
\end{align}

Now it only remains to check the initial conditions (ii). From the explicit Equation~\eqref{eq:catalyticTamariUpper}, the dominant singulariy of $f^{(0)}$, and indeed of each given $f^{(i)}$, can be computed automatically. 
We find that $f^{(0)}=\frac{64}{(\delta+3)^3}$, with singular expansion $f^{(0)}\singeq c_0 (1-t/\rho)^\alpha$ with $\alpha=\frac{1}{2}$ and $c_0=-\frac{32\sqrt{6}}{27}$. 

Given $\alpha, \beta$ we compute $\ell_0=L-0=6$, and to check the initial conditions we need to estimate the main singularity of $f^{(1)}, \dots, f^{(6)}$. This is done automatically and one gets $f^{(\ell)}\singeq c_\ell (1-t/\rho)^{\alpha-\beta\ell}$ for $\ell \geq 6$, with 
$$
c_0,\dots,c_6=
-{\frac {32}{27}}\,\sqrt {6},
{\frac {16}{27}}\,{3}^{3/4}\sqrt [4]{2},
{\frac {8}{27}},
{\frac {5}{54}}\,\sqrt [4]{3}{2}^{3/4},
{\frac {8}{81}}\,\sqrt {2}\sqrt {3},
{\frac {385}{2592}}\,{3}^{3/4}\sqrt [4]{2},
{\frac {70}{81}}.
$$

We have now verified al hypotheses of Theorem~\ref{thm:trick}. The main recurrence formula becomes 
$$
c_k=\frac{\sqrt{6}(3k-4)(3k-8)}{96}c_{k-2} \ \ , \ \  k> 6,
$$
with initial conditions above. 
It is a direct check that the solution of this recurrence is given by 
$$
c_k=
{\frac {16}{27}}\,{\frac {\Gamma  \left( \frac{k}{2}+\frac{1}{3} \right) \Gamma 
 \left( \frac{k}{2}-\frac{1}{3} \right) \sqrt {2}{4}^{-k}{6}^{3/4\,k}}{\pi }}
.
$$
for all $k\geq 0$.
Applying Theorem~\ref{thm:trick}, we immediately get:
\begin{align}\label{eq:momentstwoexpr}
		\mathbf{E}\left[ \left(\frac{Q_n(I)}{n^{3/4}}\right)^k \right]\longrightarrow 
	{\frac {\Gamma  \left( \frac{k}{2}+\frac{1}{3} \right) \Gamma  \left( \frac{k}{2}-\frac{1}{3} \right) {4}^{-k}{6}^{3/4\,k}}{\sqrt {3\pi }\Gamma  \left( \frac{3k}{4}-\frac{1}{2}\right) }}
=\frac{\sqrt{3} \cdot2^{-\frac{k}{4}-1}}{\sqrt{\pi}} \frac{\Gamma(\frac{1}{4}k+\frac{1}{3})\Gamma(\frac{1}{4}k+\frac{2}{3})}{\Gamma(\frac{1}{4}k+\frac{1}{2})},
	\end{align}
where the equality follows from  the duplication formula for the Gamma function (see also~\cite{maple}).
This proves Theorem~\ref{thm:tamari}.

\subsection{Proof of Theorem~\ref{thm:tamariLow}}

The proof is similar, we will be more succinct (all calculations are in~\cite{maple}).
We now write $f(t)=G(1,1)$. We start with Equation~\eqref{eq:exactDown} in Theorem~\ref{thm:exactDown}. 
We use this equation to compute a recurrence formula for the functions
$$
f^{(k)}\equiv f^{(k)}(t) :=\left.\left(\left(\frac{d}{dw}\right)^k G(1,1)\right)\right|_{w=1}.
$$
The main singularity of $f^{(0)}$ (which is the same as in the previous section) is at $t=\rho:=\frac{27}{256}$, which corresponds to $z=\frac{1}{4}$, i.e. $\delta=0$.

Using the package gfun, we immediately find~\cite{maple} that the $f^{(k)}$ satisfy an equation of the form~\eqref{eq:Dfinite2} with $L=9$,
where  the $h_d(t,k)$ for $d\in[9]$ are rational functions in $\delta$ with coefficients which are rational functions of $k$.
The poles of $h_1,\dots,h_9$ at $\delta=0$ are respectively of order $0, 6, 6, 10, 10, 10, 10, 10, 10$.
This implies that hypothesis (i) of Theorem~\ref{thm:trick} holds with $\beta=\frac{3}{4}$.
One can explicitly check the values of the corresponding constants $a_d(k)$, which are nothing but the coefficients of the leading term of $h_d$ in the expansion at $\delta^{-1}=\infty$, up to a scaling factor.  They are given~\cite{maple} by
\begin{align}\label{eq:explicitAi2}
	a_1(k),\dots,a_9(k) = 0, \frac{\sqrt{6}(3k-4)(3k-8)}{864}, 0, 0, 0, 0, 0, 0, 0.
\end{align}

Now it only remains to check the initial conditions (ii). From the explicit Equation~\eqref{eq:catalyticTamariUpper}, the dominant singulariy of $f^{(0)}$, and indeed of each given $f^{(i)}$, can be computed automatically. 
The function  $f^{(0)}=\frac{64}{(\delta+3)^3}$ is the same as in the previous section, with singular expansion $f^{(0)}\singeq c_0 (1-t/\rho)^\alpha$ with $\alpha=\frac{1}{2}$ and $c_0=-\frac{32\sqrt{6}}{27}$. 

We now have $\ell_0=L+0=9$, and to check the initial conditions we need to estimate the main singularity of $f^{(1)}, \dots, f^{(9)}$. This is done automatically and one gets $f^{(\ell)}\singeq c_\ell (1-t/\rho)^{\alpha-\beta\ell}$ for $\ell \leq 9$, with 
$c_\ell$ given by
\begin{center}${\footnotesize
\displaystyle c_1,...,c_9=
\frac{16\cdot 3^\frac{3}{4}2^\frac{1}{4}}{81},
\frac{8}{243},
\frac{5\cdot 2^\frac{3}{4}3^\frac{1}{4}}{1458},
\frac{8\sqrt{6}}{6561},
\frac{385\cdot 3^\frac{3}{4}2^\frac{1}{4}}{629856},
\frac{70}{59049},
\frac{85085\cdot 3^\frac{1}{4}2^\frac{3}{4}}{181398528},
\frac{700\sqrt{6}}{1594323},
\frac{37182145\cdot3^\frac{3}{4}2^\frac{1}{4}}{78364164096}.}
$
\end{center}

We have now verified al hypotheses of Theorem~\ref{thm:trick}. The main recurrence formula becomes 
$$
c_k=\frac{\sqrt{6}(3k-4)(3k-8)}{864}c_{k-2} \ \ , \ \  k> 9,
$$
with initial conditions above. 
It direct check to obtain the solution of this recurrence and, this concludes the proof of Theorem~\ref{thm:tamariLow} in the same way as we did for Theorem~\ref{thm:tamari}.
%

\subsection{Proof of Theorem~\ref{thm:mixed}}

From Theorem~\ref{thm:exactMixed} it is automatic to obtain an algebraic equation for $M(1,1)$ at $w=1$, or any of its derivatives at $w=1$ of given order. In particular one checks~\cite{maple} that the second derivative has a unique dominant singularity at $z=1/4$ and an expansion of the form
$$
\left(\frac{d}{dw}\right)^2 M(1,1) \Big|_{w=1} = O((1-4z)^{-1})
$$
at this point.
Note that this is one order of magnitude more singular (in powers of $(1-\frac{256}{27}t)\sim (1-4z)^2$) than the singularity of the main function
$$
 M(1,1) \Big|_{w=1} \singeq c (1-4z)
$$
for some $c>0$.

By transfer theorems, this immediately shows that the second moment of $\tilde{Q}_n(J)-3\tilde{P}_n(J)$ is bounded by $O(n)$. This proves~\eqref{eq:mixedSecondMoment} in Theorem~\ref{thm:mixed}.
By the Chebyshev inequality this implies that $\frac{\tilde{Q}_n(J)-3\tilde{P}_n(J)}{n^{3/4}}$ converges to zero in probability (i.e. is $o(1)$ in probability), therefore~\eqref{eq:conjThirdTilde} in Theorem~\ref{thm:mixed} follows from the previously proven convergence of $\frac{\tilde{P}_n(J)}{n^{3/4}}$ and $\frac{\tilde{Q}_n(J)}{n^{3/4}}$ to $Z$ and $Z/3$ -- note that $Z$ is almost surely nonzero. This concludes all remaining proofs.

\bigskip

The order of magnitude $O(n)$ for the second moment, as well as the explicit computation of a few more higher moments, strongly suggests that $n^{-1/2}(\tilde{Q}_n(J)-3\tilde{P}_n(J))$ has a Gaussian limit law, but we did not prove it. Theorem~\ref{thm:trick} does not immediately apply as there is an even/odd phenomenon in the form of asymptotic singularities of successive derivatives of $M(1,1)$ (this is not surprising given the fact that odd moments of the standard Gaussian are null). We suspect that the general idea of Section~\ref{sec:trick} might apply with some adaptations. However, for Gaussian limit laws there are many available tools, in particular experts in classical ACSV or quasi-power theorems might be able to prove Gaussian convergence directly from Theorem~\ref{thm:exactMixed}, while we hope experts of the Bernardi-Bonichon bijection might be able to see it directly from the central limit theorem.

In any case, we believe it is a good point to end this paper.

%

\subsection*{Acknowledgements}

 The author acknowledges funding from the grants  ANR-19-CE48-0011 ``Combiné'', ANR-18-CE40-0033 ``Dimers'', and ANR-23-CE48-0018 ''CartesEtPlus''.

\bibliographystyle{alpha}
\bibliography{biblio}

\appendix

\end{document}